%% file: jordantheory.tex
\documentclass{article}
\usepackage[margin=3.5cm]{geometry}

\input{preamble}

\DeclareFixedFont{\ttb}{T1}{txtt}{bx}{n}{8} % for bold
\DeclareFixedFont{\ttm}{T1}{txtt}{m}{n}{8}  % for normal

\usepackage{color}
\definecolor{deepblue}{rgb}{0,0,0.5}
\definecolor{deepred}{rgb}{0.6,0,0}
\definecolor{deepgreen}{rgb}{0,0.5,0}

\usepackage{listings}

\newcommand\pythonstyle{\lstset{
language=Python,
basicstyle=\ttm,
otherkeywords={self},             % Add keywords here
keywordstyle=\ttb\color{deepred},
emph={Q,to_special,normalise,iterative_reduce,__init__},          % Custom highlighting
emphstyle=\ttb\color{deepblue},    % Custom highlighting style
stringstyle=\color{deepgreen},
commentstyle=\color{deepgreen},
showstringspaces=false,
}}

\lstnewenvironment{python}[1][]
{
\pythonstyle
\lstset{#1}
}
{}

\begin{document}
\title{An algebraic semi-automated proof of the fundamental identity of Jordan algebras}
\author{John van de Wetering \\
       Radboud University Nijmegen, Netherlands\\
       \texttt{john@vdwetering.name}}
\date{\today}
\maketitle

\begin{abstract}
    The fundamental identity of quadratic Jordan algebras $Q_{Q_a b} = Q_aQ_bQ_a$ is commonly proven as a consequence of MacDonalds theorem or using more analytic methods. In this short note we give a self-contained purely algebraic proof using just a few easily proven identities and a Python script that follows a simple randomised logic to reduce expressions of Jordan operators.
\end{abstract}

\section{Introduction}
A \emph{Jordan algebra} $(V,*)$ is a vector space over a field $F$ (that we will always take to have any characteristic except 2) equipped with a bilinear operation $*$ such that for any $a,b\in V$ we have $a*b = b*a$ (commutativity) and $a*(b*a^2) = (a*b)*a^2$ (the Jordan identity) where $a^2:= a*a$. If $(V,\cdot)$ is a associative algebra (not necessarily commutative) then it becomes a Jordan algebra when equipped with the product $a*b:= \frac{1}{2}(a\cdot b + b\cdot a)$. A Jordan algebra arising from an associative algebra in this way is called \emph{special}.

For an element $a\in V$ in a Jordan algebra we define its \emph{quadratic representation} by $Q_a b:= 2a*(a*b) - a^2*b$. If $V$ is special this reduces to $Q_a b = a\cdot b\cdot a$, which is why its called quadratic. It is a well-known fact that the quadratic representation satisfies the \emph{fundamental identity}: $Q_{Q_a b} = Q_aQ_bQ_a$. This identity is readily seen to be true for special Jordan algebras as it becomes $(a\cdot b\cdot a)\cdot c \cdot (a\cdot b\cdot a) = a\cdot (b\cdot(a\cdot c\cdot a)\cdot b)\cdot a$. Many textbooks \cite{hanche1984jordan,alfsen2012geometry,chu2011jordan,mccrimmon2006taste} prove the fundamental identity as a consequence of \emph{MacDonalds theorem}. This theorem states that any polynomial identity in 3 variables that is linear in at least one variable is true for any Jordan algebra if and only if it is true for any special Jordan algebra. Other textbooks prove it when the field is the real numbers using methods from analysis \cite{faraut1994analysis}.

In this note we give a fully algebraic proof of the fundamental identity that only uses a few basic identities derived from the Jordan identity. We use a straightforward Python script (available on GitHub \cite{sourcecode} or in the source code listing below) that applies these identities as rewrite rules in a random, mostly greedy way to reduce expressions. And in this way we can show that the fundamental identity is derived. Although using a script might not be the most elegant way, it does show that you don't need much cleverness to derive the fundamental identity, and that instead it can be derived in a mostly mechanical way.

% The goal of this section will be to show that $Q_a Q_b Q_a = Q_{Q_a b}$ for arbitrary $a,b \in V$. In contrast to approaches in other works we will do this completely algebraically and without reference to MacDonalds theorem. MacDonalds theorem is a `meta-theorem' about Jordan algebras that states informally that any polynomial identity in two variables that is true for all \emph{special} Jordan algebras, is true for any Jordan algebra. Since the fundamental equality is readily verified in any special Jordan algebra, the fundamental theorem indeed follows from MacDonalds theorem. Our approach consists of three steps:
% \begin{enumerate}
%     \item First we derive an equation (see \eqref{eq:reduce}) that allows us to reduce the multiplication operator of a complex expression into a polynomial of simpler expressions. For instance $T_{a^3} = T_{a(aa)} = 3T_a^2 T_a - 2 T_a^3$. In particular, any multiplication operator of an element generated by $a$ and $b$ can be written as a polynomial in $T_a,T_b,T_{a^2}, T_{ab}$ and $T_{b^2}$.
%     \item Second, we find expressions for the commutators of these basic building blocks. With these commutation rules we can demonstrate equality of a great many equations.
%     \item Third, using a simple Python script that randomly applies the commutation rules until the entire expression is reduced to zero, we can demonstrate the fundamental equality.
% \end{enumerate}

\section{Identities}
For the rest of this note we will let $(V,*)$ be a Jordan algebra over a field $F$ not of characteristic 2 and $a,b,c,d\in V$. We let $\lambda$ always denote a scalar from $F$. We will write $ab$ for $a*b$ and $a^2b$ for $a^2*b$. We let $T_a:V\rightarrow V$ be the \emph{multiplication operator} $T_a(b):= a*b$, so that $Q_a = 2T_a^2 - T_{a^2}$. The operator $Q_a$ is a special case of the \emph{triple product} $Q_{a,b}:= T_aT_b + T_bT_a -T_{ab}$ where $Q_a = Q_{a,a}$. We let $[A,B]:= AB-BA$ denote the commutator of maps $A,B:V\rightarrow V$. Using commutativity we can recast the Jordan identity to $a*(a^2*b) = a^2*(a*b)$. Written this way it is clear that this identity is equivalent to
\begin{equation}\label{eq:jordan}
	[T_a,T_{a^2}]=0.
\end{equation}
By bilinearity we of course have $T_{a+\lambda b} = T_a + \lambda T_b$ and we can expand $(a+\lambda b)^2 = (a+\lambda b)*(a+\lambda b) = a^2 + \lambda^2 b^2 + 2\lambda ab$. With these observations in hand we can prove the \emph{linearised Jordan equations}:
\begin{lemma}
	Let $a,b,c$ be arbitrary elements of a Jordan algebra, then 
	\begin{equation}\label{eq:linjordan1}
		[T_b, T_{a^2}] + 2[T_a, T_{ab}] = 0.
	\end{equation}
	\begin{equation}\label{eq:linjordan2}
		[T_a,T_{bc}] + [T_b, T_{ac}] + [T_c, T_{ab}] = 0
	\end{equation}
\end{lemma}
\begin{proof}
	We will take the equality $[T_d,T_{d^2}] = 0$ and let $d=a\pm b$: $[T_{a\pm b},T_{(a\pm b)^2}]~=~0$. After expanding the terms we are left with
	$$[T_a, T_{a^2}] \pm [T_b, T_{b^2}] \pm \left([T_b, T_{a^2}] + 2 [T_a, T_{ab}]\right) +\left([T_a, T_{b^2}] + 2 [T_b, T_{ab}]\right) = 0.$$
	Subtracting the equation for $d=a+b$ from the equation for $d=a-b$ and dividing the result by 2 (here we use that the field is not of characteristic 2) we have the desired equation.

	We prove the second equation by taking the first equation and replacing $a$ by $a\pm c$ and using the same trick.
\end{proof}

By exploiting a symmetry in equation \eqref{eq:linjordan2} we can prove the following equation as well.
\begin{lemma}
	Let $a,b,c$ be arbitrary elements of a Jordan algebra, then
	\begin{equation}\label{eq:reduce}
		T_{a(bc)} = T_aT_{bc} + T_bT_{ac} + T_cT_{ab} - T_bT_aT_c - T_cT_aT_b.
	\end{equation}
\end{lemma}
\begin{proof}
	Apply the operators of \eqref{eq:linjordan2} to an element $d$ and bring all the negative terms to the right to get
	$$a((bc)d) + b((ac)d) + c((ab)d) = (bc)(ad) + (ac)(bd) + (ab)(cd).$$
	Observe that the righthandside is invariant under an interchange of $a$ and $d$ so that the lefthandside must be as well. This leads to the equality
	$$a((bc)d) + b((ac)d) + c((ab)d) = d((bc)a) + b((dc)a) + c((db)a) = ((bc)a)d + b(a(cd)) + c(a(bd))$$
	where we have used the commutativity of the product to move $d$ to the end in the last equality.
	Translating this back into multiplication operators, using that this equality holds for all $d$, and bringing some terms to the other side then gives the desired equation.
\end{proof}

We will refer to equation \eqref{eq:reduce} as the \emph{normalisation equation}: On the left there is a multiplication operator of a triple term $a(bc)$, while on the right there are only products of multiplication operators consisting of double terms (such as $ab$) and single terms (such as $c$). For instance by taking $a=b=c$ we can normalise $T_{a^3} = T_{a(aa)}$ to $3T_a^2 T_a - 2 T_a^3$. By repeatedly invoking this equation any product operator can be normalised to a polynomial of product operators of double and single terms.

We need to find some more commutator identities. We can already express the commutator of $T_b$ and $T_{a^2}$ and that of $T_a$ and $T_{ab}$ using equation \eqref{eq:linjordan1}. Using \eqref{eq:linjordan2} we can find the other commutators we need.

\begin{lemma}
	Let $a$ and $b$ be arbitrary elements in a Jordan algebra, then
	\begin{equation}\label{eq:commuteab_aa}
		[T_{ab}, T_{a^2}] = 2(T_a^2[T_a,T_b] + [T_a,T_b]T_a^2 -[T_aT_{a^2}, T_b])
	\end{equation}
	\begin{equation}\label{eq:commuteaa_bb}
		[T_{a^2},T_{b^2}] = 4 (T_bT_aT_bT_a - T_aT_bT_aT_b + T_aT_{ab}T_b - T_bT_{ab}T_a)
	\end{equation}
\end{lemma}
\begin{proof}
	For the first equation we start with \eqref{eq:linjordan2} with $a$ replaced by $a^2$ and $c$ by $a$. This expression involves a $T_{a^2b}$ and $T_{a^3}$ term that need to be normalised using equation \eqref{eq:reduce}. Grouping the terms in a clever way the equation is then
	$$[T_{ab},T_{a^2}] = 2 (T_a^2[T_a,T_b] + [T_a,T_b]T_a^2 - [T_aT_{a^2},T_b])+ T_a(2[T_a,T_{ab}] + [T_b,T_{a^2}])$$
	which agrees with the desired equation except for the term $T_a(2[T_a,T_{ab}] + [T_b,T_{a^2}])$ which is zero by equation \eqref{eq:linjordan1}.

	For the second equation we take \eqref{eq:linjordan1} with $b$ replaced by $b^2$ to get the equation $[T_{a^2},T_{b^2}] = 2[T_a,T_{b^2a}]$. Reducing $T_{b^2a}$ using equation \eqref{eq:reduce} the righthandside becomes 
	$$4(T_b T_a T_b T_a - T_a T_bT_a T_b - T_bT_{ab}T_a) + 2T_a([T_a,T_{b^2}] + 2 T_bT_{ab}).$$ 
	By interchanging $a$ and $b$ in equation \eqref{eq:linjordan1} we get $[T_a,T_{b^2}] + 2 T_bT_{ab} = 2T_{ab}T_b$, which when applied to the displayed equation gives the desired result.
\end{proof}

The equation $Q_{a^2} = 2L_{a^2}^2 - L_{a^2*a^2}$ can be normalised to $L_{a^2}^2 + 4 L_a^4 - 4 L_{a^2}L_a^2$. This can at a glance be seen to be equal to $Q_a^2$ by usage of the Jordan equation \ref{eq:jordan}. With a few more applications of the normalisation equation \eqref{eq:reduce} we can also derive that $Q_a^3 = Q_{a^3}$. We will now linearise these equations to get some new identities.

\begin{lemma}\label{lem:quad}
	Let $a$ and $b$ be arbitrary elements in a Jordan algebra.
	\begin{equation}\label{eq:jordanquad1}
		 4Q_{a,b}^2 = 4Q_{ab} + 2Q_{a^2,b^2} -Q_aQ_b - Q_bQ_a
	\end{equation}
	\begin{equation}\label{eq:jordanquad2}
    \begin{aligned}
        &Q_a^2Q_b + Q_aQ_bQ_a + Q_bQ_a^2 + 4Q_{a,b}^2Q_a +4Q_aQ_{a,b}^2 + 4Q_{a,b}Q_aQ_{a,b} \\
         =& Q_{a^2b} + 4Q_{a(ab)} + 4Q_{a^2b,a(ab)} + 4 Q_{a^3,b(ab)} + 2Q_{a^3,b^2a}
    \end{aligned}
	\end{equation}
	Note: $Q_{a,b}:= T_aT_b+T_bT_a - T_{ab}$ should not be confused with $Q_{ab}:= 2T_{ab}^2 - T_{(ab)^2}$.
\end{lemma}
\begin{proof}
	We know that $Q_a^2 = Q_{a^2}$. We will replace $a$ with $a+\lambda b$ in this equation. Note that $Q_{a+\lambda b}^2 = (Q_a + \lambda^2 Q_b + 2\lambda Q_{a,b})^2$ and $Q_{(a+\lambda b)^2} = Q_{a^2 + \lambda^2 b^2 + 2\lambda ab}$. The desired equation is the $\lambda^2$ term of this equation. By varying $\lambda$ we see that each term separately has to be zero which proves the equation.

	The second equation follows in exactly the same way by collecting the $\lambda^2$ term of the equation $Q_{a+\lambda b}^3 = Q_{(a+\lambda b)^3}$.
\end{proof}

\begin{theorem}\label{theor:fund}
	Let $a$ and $b$ be arbitrary elements in a Jordan algebra. $Q_{Q_a b} = Q_aQ_bQ_a$.
\end{theorem}
\begin{proof}
	First note that $Q_{Q_a b} = Q_{2a(ab) - a^2b} = Q_{a^2b} + 4Q_{a(ab)} -4Q_{a(ab),a^2b}$ and thus that the righthandside of \eqref{eq:jordanquad2} is equal to $-Q_{Q_a b} + 2Q_{a^2b} + 8 Q_{a(ab)} + 2Q_{a^3,b^2a} + 4Q_{a^3,b(ab)}$.

	On the lefthandside we will use \eqref{eq:jordanquad1} to replace both instances of $4Q_{a,b}^2$ which turns the lefthandside into
	$$-Q_aQ_bQ_a + 4(Q_aQ_{ab}+Q_{ab}Q_a) + 2(Q_aQ_{a^2,b^2} + Q_{a^2,b^2}Q_a) + 4 Q_{a,b}Q_aQ_{a,b}$$
	Equation \eqref{eq:jordanquad2} can therefore be transformed into
	\begin{align*}
		Q_aQ_bQ_a - Q_{Q_a b}  &= 4(Q_aQ_{ab}+Q_{ab}Q_a) + 2(Q_aQ_{a^2,b^2} + Q_{a^2,b^2}Q_a) + Q_{a,b}Q_aQ_{a,b} \\
		&-( 2Q_{a^2b} + 8 Q_{a(ab)} + 2Q_{a^3,b^2a} + 4Q_{a^3,b(ab)}).
	\end{align*}
	It therefore remains to show that the righthandside of this equation is zero. This can be done with a clever combination of the normalisation equation and the commutator identities as is showing using the Python script described in the next section.
\end{proof}
	
\section{Automated rewriting}
Since the final equation of the previous section is rather involved we use an automated tool to reduce it to zero. This tool takes the expression and applies certain rewrites to it. The normalisation equation \eqref{eq:reduce} gives the rewrite rule
\begin{equation}\label{rewrite:normalise}
    T_{a(bc)} \quad \rightarrow\quad T_aT_{bc} + T_bT_{ac} + T_cT_{ab} - T_bT_aT_c - T_cT_aT_b,
\end{equation}
and the Jordan equation \eqref{eq:jordan} gives the rewrites
\begin{equation}\label{rewrite:jordan}
\begin{aligned}
    T_aT_{a^2} &\quad \rightarrow \quad T_{a^2}T_a \\
    T_bT_{b^2} &\quad \rightarrow \quad T_{b^2}T_b 
\end{aligned}
\end{equation}
while equations \eqref{eq:linjordan1}, \eqref{eq:commuteab_aa} and \eqref{eq:commuteaa_bb} give the rewrite rules
\begin{equation}\label{rewrite:all}
\begin{aligned}
    T_bT_{a^2} &\quad\rightarrow\quad T_{a^2}T_b + 2T_{ab}T_a - 2T_aT_{ab} \\
        T_{a^2}T_b &\quad\rightarrow\quad T_bT_{a^2} - 2T_{ab}T_a + 2T_aT_{ab} \\
        T_{ab}T_a &\quad\rightarrow\quad T_aT_{ab} + \frac{1}{2}T_bT_{a^2} - \frac{1}{2}T_{a^2}T_b \\
        T_aT_{ab} &\quad\rightarrow\quad T_{ab}T_a - \frac{1}{2}T_bT_{a^2} + \frac{1}{2}T_{a^2}T_b \\
        T_{ab}T_{a^2} &\quad\rightarrow\quad T_{a^2}T_{ab} + 2(T_a^2[T_a,T_b] + [T_a,T_b]T_a^2 -[T_aT_{a^2}, T_b]) \\
        T_{a^2}T_{ab} &\quad\rightarrow\quad T_{ab}T_{a^2} - 2(T_a^2[T_a,T_b] + [T_a,T_b]T_a^2 -[T_aT_{a^2}, T_b]) \\
        T_{a^2}T_{b^2} &\quad\rightarrow\quad T_{b^2}T_{a^2} + 4 (T_bT_aT_bT_a - T_aT_bT_aT_b + T_aT_{ab}T_b - T_bT_{ab}T_a) \\
        T_{b^2}T_{a^2} &\quad\rightarrow\quad T_{a^2}T_{b^2} - 4 (T_bT_aT_bT_a - T_aT_bT_aT_b + T_aT_{ab}T_b - T_bT_{ab}T_a) \\
\end{aligned}
\end{equation}

The tool applies the following strategy:
\begin{enumerate}
    \item Apply the normalisation rewrite \eqref{rewrite:normalise} in no particular order until it can no longer be applied, and apply the commutation rewrites \eqref{rewrite:jordan} until they can no longer be applied.
    \item Count the number of terms in the expression. Do the following a hundred times: Randomly apply a rewrite rule from \eqref{rewrite:all} anywhere in the expression. Apply rewrites \eqref{rewrite:jordan} until they can no longer be applied. Check how many terms are left. If the amount of terms is lower than it has been since starting the set of 100 rewrites, store the current expression for the next step.
    \item If the expression has been reduced to zero we are done. If it is not zero, reset the expression to the one found in the previous step that had the least amount of terms and repeat the previous step.
\end{enumerate}

Since we have proven that all the rewrite rules come from equalities that hold in any Jordan algebra we know that if the tool manages to reduce an expression to zero that the expression must also be zero in any Jordan algebra.

A demonstration of how this works. We will first show how it is able to represent all the necessary components. Note that 'a', 'b', 'aa', 'ab', etc. are predefined objects corresponding to the terms used above, while 'Q' is a function that acts like the quadratic representation.
First, we can represent methods by how they would act on an element 'c' in a special Jordan algebra:
\begin{python}
    >>> to_special(a)
    {'ac': 0.5, 'ca': 0.5}
    >>> dict_to_expr(to_special(a))
    '0.5 ac +0.5 ca'
    >>> Q(a).to_latex()
    '2 L_aL_a - L_{a^2}'
    >>> to_special(Q(a))
    {'aca': 1.0}
\end{python}
We can use this functionality to check that our normalisations and rewrites do what we expect them to do:
\begin{python}
    >>> Q(aa).to_latex()
    '- L_{(aa)(aa)} +2 L_{a^2}L_{a^2}'
    >>> to_special(Q(aa))
    {'aacaa': 1.0}
    >>> l = Q(aa)
    >>> l.normalise()
    Normalisations applied: 2
    >>> l.to_latex()
    '4 L_aL_aL_aL_a -4 L_{a^2}L_aL_a+ L_{a^2}L_{a^2}'
    >>> to_special(l)
    {'aacaa': 1.0}
\end{python}
And we can also use it to verify that the equations we have derived are correct. For instance, recall equation \eqref{eq:jordanquad1}: $4Q_{a,b}^2 = 4Q_{ab} + 2Q_{a^2,b^2} -Q_aQ_b - Q_bQ_a$.
\begin{python}
    >>> l = 4*Q(a,b)**2
    >>> to_special(l)
    {'bbcaa': 1.0, 'abcab': 1.0, 'bacba': 1.0, 'aacbb': 1.0}
    >>> r = 4*Q(ab) + 2 * Q(aa,bb) - Q(a)*Q(b) - Q(b)*Q(a)
    >>> to_special(r)
    {'abcab': 1.0, 'bacba': 1.0, 'bbcaa': 1.0, 'aacbb': 1.0}
    >>> to_special(l) == to_special(r)
    True
\end{python}
And in fact we can prove that they are equal using the algorithm described above:
\begin{python}
    >>> h = l-r
    >>> h.normalise() #does step 1
    Normalisations applied: 5
    >>> len(h.terms)
    15
    >>> print(h)
    2 [a,a,bb] -4 [a,ab,b] +4 [a,b,a,b] +4 [a,b,ab] 
    -2 [aa,b,b]+ [aa,bb] -4 [ab,a,b] -4 [ab,b,a] 
    +12 [b,a,ab] -4 [b,a,b,a] -4 [b,aa,b] -4 [b,ab,a] 
    +6 [b,b,aa] -2 [bb,a,a]- [bb,aa]
    >>> h.iterative_reduce() #does steps 2 and 3
    New minimal amount of terms: 14
    New minimal amount of terms: 9
    New minimal amount of terms: 6
    New minimal amount of terms: 4
    New minimal amount of terms: 0
\end{python}
But as established in the previous section, the expression we really want to be able to reduce to zero is the following:
$$4(Q_aQ_{ab}+Q_{ab}Q_a) + 2(Q_aQ_{a^2,b^2} + Q_{a^2,b^2}Q_a) + Q_{a,b}Q_aQ_{a,b}
    -( 2Q_{a^2b} + 8 Q_{a(ab)} + 2Q_{a^3,b^2a} + 4Q_{a^3,b(ab)}).$$
So let's see if it works:
\begin{python}
    >>> l = (2*(Q(aa,bb)*Q(a) + Q(a)*Q(aa,bb)) +
             4*(Q(ab)*Q(a) + Q(a)*Q(ab) + Q(a,b)*Q(a)*Q(a,b)))
    >>> r = 2*Q(aa*b) + 8*Q(ab*a) + 2*Q(aa*a,bb*a) + 4*Q(aa*a,ab*b)
    >>> h = l-r
    >>> to_special(h) #verifying that this expression should indeed rewrite to zero
    {}
    >>> h.normalise()
    Normalisations applied: 93
    >>> len(h.terms)
    75
    >>> h.iterative_reduce()
    New minimal amount of terms: 73
    New minimal amount of terms: 71
    New minimal amount of terms: 70
    New minimal amount of terms: 68
    ...
    ...
    New minimal amount of terms: 11
    New minimal amount of terms: 10
    New minimal amount of terms: 8
    New minimal amount of terms: 4
    New minimal amount of terms: 0
\end{python}
So the simple algorithm described above is indeed enough to prove the fundamental equality with the reductions we have made. An obvious question to ask is if the manual rewrites we did in lemma \ref{lem:quad} and theorem \ref{theor:fund} were necessary. Couldn't we just put the fundamental equality directly in this program and derive its correctness? If we expand $Q_{Q_a b}$ as $Q_{a^2b} + 4Q_{a(ab)} -4Q_{a(ab),a^2b}$ we can represent the fundamental equality directly in the program:
\begin{python}
    >>> l = Q(aa*b)+4*Q(a*ab) - 4*Q(a*ab,aa*b)
    >>> r = Q(a)*Q(b)*Q(a)
    >>> h = l-r
    >>> h.normalise()
    Normalisations applied: 55
    >>> len(h.terms)
    59
    >>> r = h.iterative_reduce()
    New minimal amount of terms: 56
    New minimal amount of terms: 54
    ...
    New minimal amount of terms: 14
    New minimal amount of terms: 12
\end{python}
The algorithm gets stuck at 12 terms\footnote{Even though the rewrites are random, it always seems to get stuck at this particular expression.}. Scalar multiplying r by $0.5$ and printing the latex output gives:
\begin{align*}
    &L_aL_aL_aL_aL_{b^2}+ L_aL_bL_aL_bL_{a^2}- L_aL_{ab}L_bL_{a^2}+ L_bL_aL_bL_{a^2}L_a- L_bL_{ab}L_{a^2}L_a-L_{a^2}L_aL_aL_{b^2} \\
    -& L_{a^2}L_aL_bL_aL_b+ L_{a^2}L_aL_{ab}L_b- L_{a^2}L_bL_aL_bL_a+ L_{a^2}L_bL_{ab}L_a- L_{b^2}L_aL_aL_aL_a+ L_{b^2}L_{a^2}L_aL_a
\end{align*}
This is an equation that should reduce to zero, but which the algorithm doesn't seem to be able to handle. Unfortunately, I haven't been able to reduce this expression to zero by hand either, which is why the extra steps in theorem \ref{theor:fund} were necessary.
\bibliographystyle{plain}
\bibliography{bibliography}

\section{Source code listing}
Below is listed the full source code for producing the reductions above. Also available on \cite{sourcecode}.

\begin{python}
import random

# Convenience Methods
def dict_add(d,k,v):
    '''Adds values v to d[k] with a default of d[k]=0'''
    if k in d: d[k]+= v
    else: d[k] = v
def dict_remove_zeroes(d):
    '''Remove all keys of d if the value is very close to zero'''
    for k,v in list(d.items()):
        if abs(v)<0.000001:
            del d[k]
def dict_to_expr(d):
    '''takes in a dictionary where keys are strings,
    and values are numbers and prints it as a sum.
    e.g.: {"a": 1, "b": -1, "c": 2} -> a - b + 2c'''
    if not d:
        return "0"
    s = ""
    sort = sorted(d.keys())
    for term in sort:
        v = d[term]
        if abs(v-round(v))<0.00001:
            v = round(v)
        if v==1: s+= "+ " + term
        elif v==-1: s+= "- " + term
        else: s+= " {0:+} {1}".format(v, term)
    s = s.strip()
    if s[0] == "+": return s[1:].strip()
    return s

class Base(object):
    '''Base class implementing some
    boilerplate arithmetic functions'''
    def copy(self):
        return self
    def __add__(self, other):
        r = self.copy()
        r += other
        return r
    def __neg__(self):
        r = self.copy()
        return -1*r
    def __sub__(self, other):
        r = self + (-other)
        return r
    def __pow__(self, n):
        if not isinstance(n,int) or n<=0:
            raise Exception("Can only raise to a positive int")
        r = self
        for i in range(n-1):
            r = r*self
        return r
    def __hash__(self):
        return hash(str(self))
    def __eq__(self, other):
        return str(self)==str(other)
    def __repr__(self):
        return str(self)
    def to_latex(self):
        return str(self)

class JordanMonomialBase(Base):
    '''Base Class for representing the content of
    Jordan multiplication operators, e.g. 'ab' in T_{ab}'''
    def __mul__(self, other):
        return JMProduct(self, other)

    def __rmul__(self, other):
        if isinstance(other, (int,float)):
            return other * Words(self)
        if isinstance(other, JordanMonomialBase):
            return JMProduct(self, other)
        raise NotImplementedError

    def __add__(self, other):
        if isinstance(other, Words):
            other.add(1, (self,))
            return other
        return Words(self) + Words(other)

    def is_normalised(self):
        return True

class JMSingle(JordanMonomialBase):
    '''Used for representing T_a, and T_b, e.g. where
    the content of the operator is a single variable'''
    def __init__(self, variable_name):
        self._variable_name = variable_name

    def __str__(self):
        return self._variable_name

    def to_special(self):
        d = {self._variable_name: 1}
        return d

class JMProduct(JordanMonomialBase):
    '''Used for representing T_{LR} where L and R are
    other instances of JordanMonomialBase'''
    def __init__(self, L, R):
        if isinstance(L, str):
            L = JMSingle(L)
        if isinstance(R, str):
            R = JMSingle(R)

        self._L = L
        self._R = R

    def __str__(self):
        ltext = None
        rtext = None
        if isinstance(self._L, JMSingle):
            ltext = str(self._L)
        else:
            ltext = "({})".format(str(self._L))
        if isinstance(self._R, JMSingle):
            rtext = str(self._R)
        else:
            rtext = "({})".format(str(self._R))
        if ltext<rtext:
            return ltext+rtext
        else:
            return rtext+ltext

    def to_latex(self):
        s = str(self)
        if s=="aa": return "a^2"
        if s=="bb": return "b^2"
        return s

    def to_special(self):
        '''L=a, R=b, then returns 1/2ab + 1/2ba (or its dict equivalent)'''
        l = self._L.to_special()
        r = self._R.to_special()
        result = {}
        for a,t in l.items():
            for b,s in r.items(): #a and b are strings, a+b is str concat
                dict_add(result, a+b, t*s*0.5)
                dict_add(result, b+a, s*t*0.5)
        dict_remove_zeroes(result)
        return result

    def is_normalised(self):
        '''Wether the normalisation equation can be applied'''
        return (isinstance(self._L, JMSingle) and isinstance(self._R,JMSingle))

    def normalise_step(self):
        '''The product operator is of the form T_{a(bc)}. 
        Apply the normalisation rewrite and return the result'''
        if not isinstance(self._R,JMProduct):
            L = self._R
            R = self._L
        else:
            L = self._L
            R = self._R
        
        a = L
        b = R._L
        c = R._R
        result = W(a,b*c) + W(b,a*c) + W(c,a*b) - W(b,a,c) - W(c,a,b)
        return result

a = JMSingle("a")  # T_a
b = JMSingle("b")  # T_b
ab = a*b  # T_{ab}
aa = a*a  # T_{a^2}
bb = b*b  # T_{b^2}

class Words(Base):
    '''Class that can represent linear combinations of words
    of JM terms like 2T_aT_{b(ab)} - 3T_b.
    A synonym of this class is defined below as W=Words,
    because the name is used many times.'''
    def __init__(self, *term):
        self.terms = {}
        self.is_unit = True #whether the class represent the identity, e.g. whether it is empty or not
        if term:
            self.terms[Words.normalise_word(term)] = 1
            self.is_unit=False

    def __len__(self):
        return len(self.terms)

    def copy(self):
        r = Words()
        r.terms = {Words.normalise_word(k):v for k,v in self.terms.copy().items()}
        r.is_unit = self.is_unit
        return r

    @staticmethod
    def normalise_word(word):
        ''''makes sure that T_{a^2} terms always appear before T_a terms (and the same for b)'''
        r = list(word)
        while True:
            for i in range(len(word)-1):
                if ((r[i] == a and r[i+1] == aa) or
                    (r[i] == b and r[i+1] == bb)):
                    tmp = r[i]
                    r[i] = r[i+1]
                    r[i+1] = tmp
                    break
            else: #not broken out of loop, so no more rewrites
                break
        return tuple(r)
        
    def __str__(self):
        d = {}
        for k in self.terms:
            d[self._format_word(k)] = self.terms[k]
        return dict_to_expr(d)

    def _format_word(self, word):
        return "[" + ','.join([str(item) for item in word]) + "]"

    def to_latex(self):
        '''Formats the expression suitable for LaTeX output'''
        d = {}
        for k in self.terms:
            d[self._format_word_latex(k)] = self.terms[k]
        return dict_to_expr(d)

    def _format_word_latex(self,word):
        return "".join(["L_{}".format(("{"+item.to_latex()+"}") if len(item.to_latex())!=1
                                      else item.to_latex()) for item in word])

    def add(self, scalar, word):
        '''Adds scalar amount of the specified word to the expression'''
        w = Words.normalise_word(word)
        dict_add(self.terms,w,scalar)
        if abs(self.terms[w])<0.0001:
            del self.terms[w]
    
    def __iadd__(self, other):
        if other.is_unit:
            raise Exception("adding unit Words to another Words")
        for word, scalar in other.terms.items():
            self.add(scalar, word)
        return self

    def __mul__(self, other):
        if isinstance(other, JordanMonomialBase):
            other = Words(other)
        if other.is_unit: return self
        if self.is_unit: return other
        result = Words()
        result.is_unit = False
        for w1, s1 in self.terms.items():
            for w2, s2 in other.terms.items():
                result.add(s1*s2, w1+w2)
        return result

    def scalar_mult(self, val):
        for term in self.terms: self.terms[term] *= val

    def __rmul__(self,other):
        if isinstance(other, (int,float)):
            r = self.copy()
            r.scalar_mult(other)
            return r
        raise NotImplementedError

    def normalise(self):
        '''Applies the normalisation equation until it can no longer be applied'''
        amount = 0
        while True:
            word, i = self._find_normalisable_word()
            if word==None:
                print("Normalisations applied: " + str(amount))
                return
            reduced = word[i].normalise_step()
            amount += 1

            val = self.terms[word]
            reduced.scalar_mult(val)
            del self.terms[word]
            newterm = Words(*word[:i]) * reduced * Words(*word[i+1:])
            self += newterm

    def _find_normalisable_word(self):
        for word in self.terms:
            for i in range(len(word)):
                if not word[i].is_normalised():
                    return (word, i)
        return None, None

    def get_random_rewrite(self):
        possibilities = []
        t = list(self.terms.keys())
        random.shuffle(t)
        random.shuffle(rewrites)
        for term in t:
            for A,B,C in rewrites:
                if A in term and B in term:
                    #find all the matches of (A,B) in term
                    matches = [i for i in range(len(term)-1) if (A,B)==term[i:i+2]]
                    if matches:
                        return (term, random.sample(matches,1)[0], C)           
        return None

    def do_rewrite_step(self, reduc):
        term,index,rewrite = reduc
        new = W(*term[:index]) * (self.terms[term]*rewrite) * W(*term[index+2:])
        del self.terms[term]
        self += new

    def do_random_rewrites(self, N=100, list_rewrites=False, silent=False):
        '''Does a specified amount of random rewrites. Returning the expression
        with the minimal amount of terms it has found in its random path.
        If list_rewrites is True it also outputs the rewrites it has done'''
        minterms = len(self.terms)
        best_so_far = self.copy()
        path = []
        bestpath = []
        for n in range(1,N):
            reduc = self.get_random_rewrite()
            if not reduc: #no reduction possible, so we are done
                if list_rewrites: return self, bestpath
                else: return self
            self.do_rewrite_step(reduc)
            path.append(reduc)
            if len(self.terms)< minterms:
                minterms = len(self.terms)
                if not silent: print("New minimal amount of terms: " + str(minterms))
                best_so_far = self.copy()
                bestpath = path.copy()
            if n%1000 == 0:
                if not silent: print("At iteration " + str(n))

        if list_rewrites: return best_so_far,bestpath
        else: return best_so_far
        
    def iterative_reduce(self, iterations=100):
        '''For a specified amount of iterations, do 100 random reductions and remember
         the shortest expression. Then repeat the process with this expression'''
        r = self.copy()
        shortest = len(self.terms)
        for i in range(iterations):
            r = r.do_random_rewrites(100)
            if len(r) < shortest:
                shortest = len(r)
                if len(r)==0: break
        return r

    def greedy_reduce(self):
        '''Tries 500 random reductions, and remembers the best one.
        It keeps doing this to reduce the expression as 'greedily' as possible.
        If it gets stuck, it tries to do 2 reductions at once, and then 3, etc.
        It returns the set of reductions it has found.
        '''
        path = []
        r = self.copy()
        r.normalise()
        temperature = 1
        while len(r.terms):
            reductions = []
            best = r
            length = len(r.terms)
            for i in range(500):
                r2 = r.copy()
                r3, reduc = r2.do_random_rewrites(temperature,True,True)
                if len(r3.terms)<length:
                    reductions = reduc
                    best = r3
                    length = len(r3.terms)
            if reductions:
                print("found new step. {!s} terms left".format(length))
                r = best
                path.extend(reductions)
                temperature = 1
            else:
                temperature += 1
            
        return path

W = Words

def Q(a,b=None):
    if b: return W(a,b) + W(b,a) - W(a*b)
    return 2*W(a,a) - W(a*a)

# list of tuples (A,B,C) with the understanding that AB
# may be interchanged with BA + C or BA with AB - C
rewrite_base = [
    (b, ab, 0.5*(W(bb,a) - W(a,bb))),
    (a, ab, 0.5*(W(aa,b) - W(b,aa))),
    (a, bb, 2*(W(ab,b) - W(b,ab))),
    (b, aa, 2*(W(ab,a) - W(a,ab))),
    (ab,aa, -2*(W(b,a,a,a)-W(a,a,a,b))-2*(W(aa,a,b)-W(b,aa,a))-2*(W(a,a,b,a)-W(a,b,a,a))),
    (ab,bb, -2*(W(a,b,b,b)-W(b,b,b,a))-2*(W(bb,b,a)-W(a,bb,b))-2*(W(b,b,a,b)-W(b,a,b,b))),
    (aa,bb, 4*(W(b,a,b,a) - W(a,b,a,b) + W(a,ab,b) - W(b,ab,a))),
    ]

# add reverse rewrites
r = []
for A,B,C in rewrite_base:
    r.append((B,A, -C))

rewrite_base.extend(r)
del r
rewrites = []
for A,B,C in rewrite_base:
    rewrites.append((A, B, W(B,A) + C))

def verify_rewrites():
    for A,B, l in rewrite_base:
        h = W(A,B)-W(B,A) - l
        print(to_special(h))

def to_special(l,name="c"):
    if not isinstance(l, W):
        l = W(l)
        
    result = {}
    for word in l.terms:
        w = list(reversed(word))
        r = {name:1}
        for t in w:
            r2 = {}
            for a, s1 in t.to_special().items():
                for k,v in r.items():
                    dict_add(r2,a+k,0.5*v*s1)
                    dict_add(r2,k+a,0.5*v*s1)
            r = r2
        for k,v in r.items():
            dict_add(result, k, v*l.terms[word])
            
    dict_remove_zeroes(result)
    return result

def fundamental_ident():
    l = l = Q(aa*b)+4*Q(a*ab) - 4*Q(a*ab,aa*b) # Q_{Q_a b}
    r = Q(a)*Q(b)*Q(a)
    h = l-r
    print(to_special(h))  # Verify h should indeed be reducable to zero
    return h

def equation7():
    l = 4*Q(a,b)**2
    r = 4*Q(ab) + 2*Q(aa,bb) - Q(a)*Q(b) - Q(b)*Q(a)
    h = l-r
    print(to_special(h))
    return h
def equation8(): # Equation (8)
    l = Q(a)*Q(b)*Q(a) + Q(aa)*Q(b) + Q(b)*Q(aa)\
        + 4*(Q(a,b)*Q(a,b)*Q(a)+Q(a,b)*Q(a)*Q(a,b) + Q(a)*Q(a,b)*Q(a,b))
    r = Q(aa*b)+4*Q(ab*a) + 4*Q(aa*b,ab*a) + 2*Q(aa*a,bb*a) + 4*Q(aa*a,ab*b)
    h = l-r
    print(to_special(h))
    return h
def theorem(): # Reducing this to zero proves the fundamental equality
    l = (2*(Q(aa,bb)*Q(a) + Q(a)*Q(aa,bb)) 
         + 4*(Q(ab)*Q(a) + Q(a)*Q(ab) + Q(a,b)*Q(a)*Q(a,b)))
    r = 2*Q(aa*b) + 8*Q(ab*a) + 2*Q(aa*a,bb*a) + 4*Q(aa*a,ab*b)
    h = l-r
    print(to_special(h))
    return h

def theorem2(): # Reducing this to zero proves the fundamental equality - but with less manual work.
    l = (Q(aa)*Q(b) + Q(b)*Q(aa) + 4*(Q(a,b)**2*Q(a)
        + Q(a,b)*Q(a)*Q(a,b) + Q(a)*Q(a,b)**2))
    r = 8*Q(aa*b,a*ab) + 2*Q(aa*a,bb*a) + 4*Q(aa*a,ab*b)
    h = l-r
    print(to_special(h))
    return h

\end{python}

\end{document}

%% file: preamble.tex
\usepackage[utf8]{inputenc}
\usepackage[english]{babel}
\usepackage{amsmath}
\usepackage{amsfonts}
\usepackage{amsthm}
\usepackage{graphicx}
\usepackage{physics}
\usepackage{mathtools}
\usepackage{xspace}
\usepackage{tikz-cd}
\usepackage{hyperref}
\usepackage[toc,page]{appendix}

\usepackage{enumerate}

\theoremstyle{definition}
\newcounter{counter}
\numberwithin{counter}{section}
\newtheorem{theorem}[counter]{Theorem}

\newtheorem{lemma}[counter]{Lemma}

\newtheorem*{theorem*}{Theorem}
\newtheorem*{lemma*}{Lemma}

%% file: jordantheory.bbl
\begin{thebibliography}{1}

\bibitem{alfsen2012geometry}
Erik~M Alfsen and Frederic~W Shultz.
\newblock {\em Geometry of state spaces of operator algebras}.
\newblock Springer Science \& Business Media, 2012.

\bibitem{chu2011jordan}
Cho-Ho Chu.
\newblock {\em Jordan structures in geometry and analysis}, volume 190.
\newblock Cambridge University Press, 2011.

\bibitem{faraut1994analysis}
Jacques Faraut and Adam Kor{\'a}nyi.
\newblock {\em Analysis on symmetric cones}.
\newblock Clarendon Press Oxford, 1994.

\bibitem{hanche1984jordan}
Harald Hanche-Olsen and Erling St{\o}rmer.
\newblock {\em Jordan operator algebras}, volume~21.
\newblock Pitman Advanced Pub. Program, 1984.

\bibitem{mccrimmon2006taste}
Kevin McCrimmon.
\newblock {\em A taste of Jordan algebras}.
\newblock Springer Science \& Business Media, 2006.

\bibitem{sourcecode}
John van~de Wetering.
\newblock Jordan Algebra Rewriting Script.
\newblock \url{https://github.com/johnie102/jordanrewrite}, 2018.

\end{thebibliography}
